\documentclass{article}
\usepackage{calc}
\usepackage[shortlabels]{enumitem}

\setlist[enumerate]{label=\textup{(\alph*)}, leftmargin=\widthof{(a)}+\the\labelsep}
\setlist[itemize]{leftmargin=\widthof{$\bullet$}+\the\labelsep}

\usepackage[margin=1in]{geometry}
\usepackage{amsmath, amsfonts, amsthm, amssymb}
\usepackage{url}

\DeclareMathOperator{\Gal}{Gal}
\newcommand{\GL}{\mathrm{GL}}
\newcommand{\PGL}{\mathrm{PGL}}
\newcommand{\GA}{\mathrm{GA}}

\newcommand{\FF}{\mathbb{F}}

\newcommand{\ZZ}{\mathbb{Z}}
\newcommand{\II}{\mathcal{I}}
\newcommand{\fP}{\mathfrak{P}}
\newcommand{\fS}{\mathfrak{S}}
\newcommand{\bs}{\backslash}
\newcommand{\cross}{\times}
\newcommand{\size}[1]{\lvert #1 \rvert}
\newcommand{\Size}[1]{\left\lvert #1 \right\rvert}

\newcommand{\union}{\cup}
\newcommand{\isom}{\cong}
\newcommand{\floor}[1]{\lfloor #1 \rfloor}
\newcommand{\Floor}[1]{\left\lfloor #1 \right\rfloor}
\newcommand{\ignore}[1]{}

\theoremstyle{theorem}
\newtheorem{theorem}{Theorem}
\newtheorem{lem}{Lemma}
\newtheorem{cor}{Corollary}
 
\theoremstyle{definition}
\newtheorem*{definition}{Definition}

\begin{document}
\title{Visibly irreducible polynomials over finite fields}
\markright{Visibly irreducible polynomials}
\author{Evan M. O'Dorney}
\maketitle
\begin{abstract}
Lenstra, in this \textsc{Monthly}, has pointed out that a cubic over $\FF_5 = \ZZ/5\ZZ$ of the form $(x-a)(x-b)(x-c) + \lambda(x-d)(x-e)$, where $\{a,b,c,d,e\}$ is some permutation of $\{0,1,2,3,4\}$, is irreducible because every element of $\FF_5$ is a root of one summand but not the other. We classify polynomials over finite fields that admit an irreducibility proof with this structure.
\end{abstract}

\section{Introduction.}
In a past note in this \textsc{Monthly} \cite{Lenstra5}, Lenstra relates how he was trying to set an examination problem of a standard genre---namely, factoring a polynomial over a finite field---whose answer could be verified by a quick, humanly comprehensible argument. He chose the following polynomial:
\begin{equation*}
  f(x) = x^3 - 3x^2 - x - 3 \in \FF_5[x].
\end{equation*}
(Here and throughout this article, $\FF_q$ denotes the field with $q$ elements.) Built in was the following solution:
\begin{equation} \label{eq:f partic}
  f(x) = (x^3 - x) - (3x^2 + 3) = x(x + 1)(x - 1) - 3(x + 2)(x - 2),
\end{equation}
which shows that $f$ is in fact irreducible: for if it factored, it would have to have a linear factor, and each of the five possible linear factors over $\FF_5$ divides one but not the other of the two summands of \eqref{eq:f partic}. The same proof applies to any polynomial over $\FF_5$ of the form
\begin{equation} \label{eq:cubics f5}
  (x - a)(x - b)(x - c) + \lambda (x - d)(x - e),
\end{equation}
where $\{a,b,c,d,e\}$ are the elements $\{0,1,2,3,4\} = \FF_5$ in some order and $\lambda \in \FF_5^\cross$ is a nonzero constant. Lenstra proves that every monic irreducible cubic over $\FF_5$ has the form \eqref{eq:cubics f5} in a unique way, up to permuting the factors in each term, and gives a pleasant algorithm for finding $a$, $b$, $c$, $d$, $e$, and $\lambda$.

In this article we address the natural question (\cite{Lenstra5}, p.~818) of the extent to which this phenomenon extends to other degrees of polynomials and other fields.

We can say at once that the phenomenon is not restricted to $\FF_5$. Taking $\FF_2$, the simplest of all fields, and writing the quadratic polynomial
\begin{equation*}
  f(x) = x^2 + x + 1 = (x)^2 + (x + 1),
\end{equation*}
we see that $f$ is irreducible, because the two possible linear factors $x, x + 1$ each divide one but not the other of the two terms of the decomposition. This is not the only such irreducibility proof for this polynomial: equally effective are
\begin{equation*}
  f(x) = (x^2 + 1) + x = (x+1)^2 + x
\end{equation*}
and
\begin{equation*}
  f(x) = (x^2 + x) + 1 = x(x+1) + 1.
\end{equation*}
The same argument applies to cubics over $\FF_2$ such as
\begin{equation*}
  f(x) = x^3 + x + 1 = (x)^3 + (x + 1) = x(x+1)^2 + 1 = \cdots.
\end{equation*}
But if we write a quartic in a form like
\begin{equation*}
  f(x) = x^4 + x + 1 = (x)^4 + (x + 1),
\end{equation*}
the irreducibility is no longer clear. We see that $f$ has no linear factor, but a quartic polynomial could still factor as the product of two quadratics. If we know (somehow) that $x^2 + x + 1$ is the only irreducible quadratic over $\FF_2$, then we can write
\begin{equation*}
  f(x) = x^4 + x + 1 = x(x + 1)(x^2 + x + 1) + 1,
\end{equation*}
and now it \emph{is} visible that $f$ is not divisible by either of the linear factors $x$, $x + 1$ or the quadratic factor $x^2 + x + 1$, and hence $f$ must be irreducible.

Motivated by the foregoing examples, we make the following definition:

\begin{definition}
A \emph{visibly irreducible decomposition (VID)} of degree $d \geq 2$ over the finite field $\FF_q$ is a sum $f_1(x) + f_2(x) + \cdots + f_r(x) = f(x)$ of $r \geq 2$ nonzero polynomials $f_i(x) \in \FF_q[x]$ of degree at most $d$ with the following properties:
\begin{enumerate}[(V{I}D-1), leftmargin=\widthof{(V{I}D-1)}+\the\labelsep]
  \item\label{it:vid div} Every irreducible polynomial $p(x)$ of degree not exceeding $d/2$ (known as an \emph{operative factor}) divides all but exactly one of the $f_i(x)$. This makes it visible that $p(x) \nmid f(x)$.
  \item\label{it:vid deg} Exactly one of the $f_i$ actually has degree $d$, the others having degree less than $d$. This makes it visible that $f$ has degree exactly $d$.
\end{enumerate}
\end{definition}
Condition \ref{it:vid deg} may seem a bit arbitrary, but it ensures that the sum $f(x)$ has degree exactly $d$, without the need to check the sum of the leading coefficients. Without it, $f(x)$ could be an irreducible polynomial of any degree from $\floor{d/2} + 1$ to $d$ inclusive---or could be a constant! Condition \ref{it:vid deg} is also motivated by symmetry considerations, as will be explained in Section \ref{sec:symmetry}. At the end of this article we will briefly note what happens if it is removed.

Our main result is the following determination of which polynomials admit a VID.

\begin{theorem} \label{thm:main}~
\begin{enumerate}
\item \label{it:main all} For the following pairs $(q,d)$, \textsc{every} irreducible polynomial of degree $d$ over $\FF_q$ admits a VID:
\begin{itemize}
  \item $(2,2)$, $(3,2)$
  \item $(2,3)$, $(3,3)$, $(4,3)$, $(5,3)$
  \item $(2,4)$
  \item $(2,5)$
  \item $(2,6)$
  \item $(2,7)$.
\end{itemize}
\ignore{That is to say,
\begin{itemize}
  \item Quadratics over $\FF_2$ and $\FF_3$
  \item Cubics over $\FF_2$, $\FF_3$, $\FF_4$, and $\FF_5$
  \item Quartics over $\FF_2$
  \item Quintics over $\FF_2$
  \item Sextics over $\FF_2$
  \item Septimics over $\FF_2$.
\end{itemize}
}
\item \label{it:main half}
For $(q,d) = (3,5)$, exactly \textsc{half} of all irreducible quintics over $\FF_3$ admit a VID.
\item \label{it:main none}
For all other $q$ and $d$, \textsc{no} irreducible polynomial admits a VID.
\end{enumerate}
\end{theorem}

\section{No VID's for large fields or high degrees.} We begin with the proof of Theorem \ref{thm:main}\ref{it:main none}, which restricts the $(q,d)$ pairs to be considered to a finite list. The method is quite straightforward.
\label{sec:bounds}

\begin{lem} \label{lem:basic bd}
If a VID $f_1 + \cdots + f_r$ of degree $d$ exists over $\FF_q$, then
\begin{equation} \label{eq:basic bd}
  dr \geq (r-1) \left[1 + \bigg\lvert \bigcup_{n=1}^{\floor{d/2}} \FF_{q^n} \bigg\rvert \right],
\end{equation}
where the union is taken within the algebraic closure $\overline{\FF_q}$ (which contains a unique isomorphic copy of $\FF_{q^n}$ for each $n$).
\end{lem}
\begin{proof}
Let $\xi \in \FF_{q^n}$, $1 \leq n \leq \floor{d/2}$. The minimal polynomial $p(x)$ of $\xi$ is irreducible of degree $n \leq \floor{d/2}$ and thus must divide all but one of the $f_i$; therefore $\xi$ is a root of the product $f_1f_2\cdots f_r$ of multiplicity at least $r-1$. But this is a product of one factor of degree $d$ and $r-1$ factors of degree at most $d-1$, so
\begin{equation*}
  d + (d-1)(r-1) \geq (r-1) \bigg\lvert \bigcup_{n=1}^{\floor{d/2}} \FF_{q^n} \bigg\rvert,
\end{equation*}
which simplifies to \eqref{eq:basic bd}.
\end{proof}

\begin{lem} \label{lem:when bd}
The bound \eqref{eq:basic bd} can hold only for the pairs $(q,d)$ mentioned in Theorem \ref{thm:main}\ref{it:main all},\ref{it:main half}.
\end{lem}
\begin{proof}
The bound is weakest when $r = 2$, so it suffices to determine when it can hold in this case. We have
\begin{equation} \label{eq:q_bound}
  2d \geq 1 + \bigg\lvert \bigcup_{n=1}^{\floor{d/2}} \FF_{q^n} \bigg\rvert \geq 1 + \Size{ \FF_{q^{\floor{d/2}}} } = 1 + q^{\floor{d/2}}.
\end{equation}
In particular,
\begin{equation*}
  2d \geq 1 + 2^{\floor{d/2}}
\end{equation*}
which is seen to hold only when $d \leq 7$ or $d = 9$. But the $d = 9$ case, upon substituting back into \eqref{eq:basic bd}, yields
\begin{equation*}
  18 = 2d \geq 1 + \Size{ \FF_{q^3} \union \FF_{q^4} } = 1 + q^3 + q^4 - q \geq 23,
\end{equation*}
which is untrue. So we have $2 \leq d \leq 7$. For each $d$, \eqref{eq:q_bound} bounds the value of $q$ by
\begin{equation*}
  q \leq \Floor{(2d - 1)^{1/\floor{d/2}}}.
\end{equation*}
Plugging $d = 2,3,4,5,6,7$ into this fanciful-looking expression yields the bounds of $3$, $5$, $2$, $3$, $2$, and $2$ respectively, precisely as desired.
\end{proof}

% section bounds (end)

\section{Symmetry.}
\label{sec:symmetry}
Proving Theorem \ref{thm:main}\ref{it:main all} is a finite problem: for each $(q,d)$, there are a finite number of irreducibles, and we simply need to write a VID for each one! However, throughout this article, we will strive to prove results conceptually rather than resorting to computation. In this section, we describe a family of symmetries that allow us to consider only a small number of irreducibles per $(q,d)$ pair.

The symmetries are best described in terms of \emph{homogeneous forms} of degree $d$ in two variables $X,Y$. These are in bijection with one-variable polynomials of degree at most $d$, via the standard operations of \emph{homogenization}
\begin{equation*}
  f(x) \mapsto F(X,Y) = Y^d f(X/Y)
\end{equation*}
and \emph{dehomogenization}
\begin{equation*}
  F(X,Y) \mapsto f(x) = F(x,1),
\end{equation*}
and we will frequently identify one-variable polynomials with their homogenizations.

In the homogeneous context, we have the following attractive notion of VID:
\begin{definition}
A \emph{(homogeneous) VID} of degree $d$ over a finite field $\FF_q$ is a sum $F_1(X,Y) + F_2(X,Y) + \cdots + F_r(X,Y) = F(X,Y)$ of $r \geq 2$ nonzero homogeneous forms of degree $d$ over $\FF_q$, satisfying a single property:
\begin{enumerate}[(HV{I}D), leftmargin=\widthof{(HV{I}D)}+\the\labelsep]
  \item\label{it:vid hom} Every irreducible homogeneous form $P(X,Y)$ of degree not exceeding $d/2$ (called an \emph{operative factor}) divides all but exactly one of the $F_i(x)$. This makes it visible that $P(X,Y) \nmid F(X,Y)$.
\end{enumerate}
\end{definition}
We see that \ref{it:vid hom} for each operative factor $P(X,Y)$ corresponds to \ref{it:vid div} for the corresponding inhomogeneous operative factor $p(x)$, \emph{except} for the special operative factor $P(X,Y) = Y$, for which \ref{it:vid hom} corresponds to \ref{it:vid deg}. Thus this notion of VID is entirely compatible with the one above.

Now, the group $\GL_2(\FF_q)$ acts on homogeneous forms of degree $d$ by linear change of variables
\begin{equation*}
  \begin{bmatrix}
    \alpha & \beta \\ \gamma & \delta
  \end{bmatrix}
  \cdot F(X,Y) = F(\alpha X + \gamma Y, \beta X + \delta Y).
\end{equation*}
The scalar matrix $\alpha I$ acts by multiplication by $\alpha^d$, and thus the quotient $\Gamma = \PGL_2(\FF_q)$ acts on the set of forms of degree $d$ up to scaling. Moreover, the $\Gamma$-action preserves irreducibility and acts on VID's of each degree up to scaling (where \emph{scaling} a VID means scaling all its summands $F_i$ by a single scalar $\alpha \in \FF_q^\cross$).

Though we will not need it, the $\Gamma$-action can be described directly on inhomogeneous polynomials as
\begin{equation*}
  \begin{bmatrix}
    \alpha & \beta \\ \gamma & \delta
  \end{bmatrix}
  \cdot f(x) = (\beta x + \delta)^d f\left(\frac{\alpha x + \gamma}{\beta x + \delta}\right);
\end{equation*}
in such form it was studied in \cite{GarAction}.

Let $\II(q,d)$ be the set of irreducible homogeneous forms of degree $d$ over $\FF_q$, up to scaling. These are in bijection with the monic irreducible one-variable polynomials of degree $d$ if $d \geq 2$. The size of $\II(q,d)$ is given by the classical formula (due to Gauss in the case $q$ prime; see \cite{CheboluGauss} for a simple proof in the general case):
\begin{equation} \label{eq:num irreds}
  \Size{\II(q,d)} = \begin{cases}
    \displaystyle \frac{1}{d} \sum_{k|d} \mu(k) q^{d/k}, & d \geq 2, \\
    q + 1, &d = 1
  \end{cases}
\end{equation}
where $\mu(k)$ is the M\"obius function. If one $F \in \II(q,d)$ admits a VID, then so do all irreducibles in the $\Gamma$-orbit of $F$. Therefore we will begin by counting the $\Gamma$-orbits on $\II(q,d)$. We begin with a pair of simple results.
\begin{lem}\label{lem:sim tran}
~
\begin{enumerate}
\item \label{it:st 2} The group $\GA_1(\FF_q)$ of affine transformations of the line (which is also the subgroup of transformations in $\Gamma$ fixing one linear form $Y \in \II(q,1)$) acts simply transitively on $\FF_{q^2} \bs \FF_q$.
\item \label{it:st 3} $\Gamma$ acts simply transitively (by linear fractional transformations) on $\FF_{q^3} \bs \FF_q$.
\end{enumerate}
\end{lem}
\begin{proof} The proof method in each case is similar:
\begin{enumerate}
\item Since $\Size{\GA_1(\FF_q)} = q(q-1) = \Size{\FF_{q^2} \bs \FF_q}$, it is enough to prove that the stabilizer of each point is trivial. Suppose $\gamma \in \GA_1(\FF_q)$ fixes $\xi \in \FF_{q^2} \bs \FF_q$. Then $\gamma$ also fixes $\tau(\xi)$, where $\tau \in \Gal(\FF_{q^2}/\FF_q)$ is a generator. Since $\xi \notin \FF_q$, we have $\xi \neq \tau(\xi)$. An affine transformation that fixes two points must be the identity.
\item Since $\size{\Gamma} = q(q-1)(q+1) = \Size{\FF_{q^3} \bs \FF_q}$, it is enough to prove that the stabilizer of each point is trivial. Suppose $\gamma \in \Gamma$ fixes $\xi \in \FF_{q^3} \bs \FF_q$. Then $\gamma$ also fixes $\tau(\xi)$ and $\tau^2(\xi)$, where $\tau \in \Gal(\FF_{q^3}/\FF_q)$ is a generator. Since $\xi \notin \FF_q$, the three conjugates $\xi$, $\tau(\xi)$, $\tau^2(\xi)$ are distinct. A linear fractional transformation that fixes three points must be the identity. \qedhere
\end{enumerate}
\end{proof}

\begin{lem} \label{lem:one_orbit}
The values of $(q,d)$ for which $\II(q,d)$ consists of a single $\Gamma$-orbit are as follows: all $d \leq 3$, and $(2,4)$ and $(2,5)$.
\end{lem}
The $\Gamma$-action on $\II(q,d)$ is well studied: formulas have been published for the number of fixed points of various elements and subgroups of $\Gamma$ \cite{AhmadiQuadratic, Carlitz, GarAction}. Nevertheless, no one in the literature seems to have posed before the simple question of when $\II(q,d)$ is a single $\Gamma$-orbit. 

\begin{proof}[Proof of Lemma \ref{lem:one_orbit}]
By Lemma \ref{lem:sim tran}, $\Gamma$ transitively permutes the elements of $\FF_{q^2} \bs \FF_q$ (respectively, $\FF_{q^3} \bs \FF_q$) and thus also transitively permutes their minimal polynomials, which comprise $\II(q,2)$ (respectively, $\II(q,3)$). Here we are using that the minimal polynomial of $\xi \in \FF_{q^d}$ is, upon homogenization, the lowest-degree form defined over $\FF_q$ divisible by $X - \xi Y$; and $\Gamma$ acts on these linear forms up to scaling as it does on the elements $\xi\in \FF_{q^d} \union \{\infty\}$ via linear fractional transformations. (The matrix
\[
  g = \begin{bmatrix}
    \alpha & \beta \\ \gamma & \delta
  \end{bmatrix}
\]
does not act by $\xi \mapsto \frac{\alpha\xi + \beta}{\gamma\xi + \delta}$, as is natural, but by $\xi \mapsto \frac{\delta\xi - \gamma}{-\beta\xi + \alpha}$; but this is of no significance for the orbits.) This takes care of the cases where $d \leq 3$.

In the cases $(q,d) = (2,4)$ and $(2,5)$, we can prove the lemma by bounding the point stabilizers in the following standard way, which will also be important later:
\begin{lem} \label{lem:stab}
If $d \geq 3$, then for every $F(X,Y) \in \II(q,d)$, the stabilizer $\Gamma_F$ is a cyclic group of order dividing $d$.
\end{lem}
\begin{proof}
Let $f(x)$ be the dehomogenization of $F$. The stabilizer $\Gamma_F = \Gamma_f$ permutes the roots of $f$ in $\FF_{q^d}$ and thus maps naturally into $\Gal(\FF_{q^d} / \FF_q) \isom C_d$, a cyclic group. Any $\gamma \in \Gamma_f$ that fixes one root must fix all the roots, and since there are at least three roots, we must in fact have $\gamma = I$. So $\Gamma_f$ maps isomorphically to a subgroup of $C_d$.
\end{proof}
In the cases $(q,d) = (2,4)$ and $(2,5)$, we get that all orbits have size
\begin{equation*} %\label{eq:stab}
  \frac{\size{\Gamma}}{\Size{\Gamma_f}} \geq \frac{\size{\Gamma}}{\gcd(\size{\Gamma}, d)} = \begin{cases}
    3, \quad d = 4 \\
    6, \quad d = 5.
  \end{cases}
\end{equation*}
By Gauss's formula \eqref{eq:num irreds}, there are exactly $3$ irreducible quartics and $6$ irreducible quintics over $\FF_2$, implying that there is only a single orbit in these cases.

For all other $(q,d)$, there is \emph{more} than one orbit. For $d \geq 5$ this can be seen simply by showing that
\begin{equation*} %\label{eq:more orbits}
  \size{\II(q,d)} > \size{\Gamma},
\end{equation*}
an exercise in bounding. For $d = 4$, we claim that the point stabilizer of any irreducible $f \in \II(q,4)$ has order at least $2$. Consider the permutation of the roots of $f$ given by the square of Frobenius: $\tau(\xi) = \xi^{q^2}$. As this permutation is in the Klein four group, it preserves the cross ratio of the four roots (the unique invariant
\[
  \frac{(\xi_1 - \xi_2)(\xi_3 - \xi_4)}{(\xi_1 - \xi_3)(\xi_2 - \xi_4)}
\]
of quadruples of distinct points on the projective line) and thus is given by a linear fractional transformation, which, being characterized in a Galois-invariant way, is defined over $\FF_q$. So after verifying the weaker inequality
\begin{equation*}
  \size{\II(q,4)} > \frac{\size{\Gamma}}{2}
\end{equation*}
for $q \geq 3$, it follows that there is more than one orbit.
\end{proof}

% section symmetry (end)

\section{Construction of VID's.} \label{sec:examples}
Writing a VID of a given degree is a quite intuitive matter: we place each operative factor in the appropriate summands and repeat factors (or, in rare cases, add higher-degree factors) to bring the total degree of each term up to $d$. For instance, in the case $(q,d) = (4,3)$, $r = 2$, the operative factors are five linear forms $L_1,\ldots, L_5$. Because the degree of each term cannot exceed $3$, they must appear in the distribution
\begin{equation*}
  L_1 L_2 + L_3 L_4 L_5;
\end{equation*}
then, to bring the degree of the first term up to $3$, we add another factor of $L_1$ or $L_2$ (not $L_3$, $L_4$, or $L_5$!) to get the VID shape
\begin{equation*}
  L_1^2 L_2 + \alpha L_3 L_4 L_5,
\end{equation*}
where the relative scaling $\alpha$ as well as the ordering of the forms $L_i$ can freely vary. For clarity's sake we include a formal exposition.
\begin{definition}
A \emph{shape} of degree $d$ is a sum
\begin{equation*}
  \fS = \sum_{i = 1}^r \left( \alpha_i \prod_{j = 1}^{s_i} \fP_{ij} \right)
\end{equation*}
of products of formal factors $\fP_{ij}$, with two attached pieces of data:
\begin{enumerate}
  \item a positive integer $\deg \fP_{ij}$ for each factor, to be thought of as a degree, with each summand having total degree $\sum_j \deg \fP_{ij} = d$;
  \item an equivalence relation $\fP_{ij} \equiv \fP_{k\ell}$ among the factors, respecting degree (that is, such that two equivalent factors have the same degree).
\end{enumerate}
\end{definition}

\begin{definition}
An \emph{instance} of a shape over a field $\FF$ is an actual sum $\sum_i F_i$ given by replacing each formal factor $\fP_{ij}$ with an actual homogeneous polynomial of the specified degree over $\FF$, so that equivalent factors get replaced by the same (or proportional) forms and inequivalent factors by nonproportional forms, and specifying the relative scalings $\alpha_i$ of the terms. Rescaling the entire sum, or fiddling with the scalings of each factor without changing the overall relative scalings of each term, will be considered to yield the same instance. Permuting the terms $F_i$ will also be considered to yield the same instance, if the shape happens to be invariant under some such permutation.
\end{definition}

\begin{definition}
A \emph{visibly irreducible shape (VIS)} of degree $d$ over $\FF_q$ is a shape of degree $d$ in which, for $1 \leq n \leq d/2$, there are exactly $\size{\II(q,d)}$ inequivalent factors of degree $n$ and each appears in all but one summand of the shape.
\end{definition}

These definitions have been arranged to make it obvious that every VID is an instance of a unique VIS, and every instance of a VIS is a VID. We now proceed with the construction.

\subsection{The one-orbit cases.} The cases $(q,d)$ where $\II(q,d)$ is a single $\Gamma$-orbit are the simplest to analyze. One simply has to write a single VIS $\fS$; then its instances (being a $\Gamma$-invariant set) represent all irreducibles in $\II(q,d)$. Moreover, each irreducible is represented the same number of times, which may readily be computed by dividing the number of instances of $\fS$ by $\size{\II(q,d)}$.

Considerations of space prevent us from classifying \emph{all} VIS's, though such a classification is certainly within reach; we limit ourselves to listing one VIS per $(q,d)$ pair. We write shapes as follows: $L$, $Q$, and $C$ with possible subscripts denote linear, quadratic, and cubic factors respectively, those with different subscripts being inequivalent. The formal coefficient $\alpha_i$ of one term can be suppressed, and all the $\alpha_i$ can be suppressed if $q = 2$.

\begin{table}[htbp]
\begin{tabular}{cc|l|l|cc} \\
$q$ & $d$ & $r$ & Example VIS & $\size{\II(q,d)}$ &
  \begin{tabular}{@{}c@{}}
    \# of VID's \\ of this shape \\ per irred
  \end{tabular} \\ \hline
$2$ & $2$ & $2,3$ & $L_1 L_2 + L_2 L_3 + L_3 L_1$ & $1$ & $1$ \\
$2$ & $3$ & any & $L_1^2 L_2 + L_2^2 L_3 + L_3^2 L_1$ & $2$ & $1$ \\
$2$ & $4$ & $2,3,4$ & $L_1^2 L_2 L_3 + Q^2$ & $3$ & $1$ \\
$2$ & $5$ & any & $L_1^4 L_2 + L_3 Q^2$ & $6$ & $1$ \\
$3$ & $2$ & $2$ & $L_1 L_2 + \alpha L_3 L_4$ & $3$ & $2$ \\
$3$ & $3$ & $2,3,4$ & $L_1^3 + \alpha L_2 L_3 L_4$ & $8$ & $1$ \\
$4$ & $3$ & $2$ & $L_1^2 L_2 + \alpha L_3 L_4 L_5$ & $20$ & $3$ \\
$5$ & $3$ & $2$ & $L_1 L_2 L_3 + \alpha L_4 L_5 L_6$ & $40$ & $1$
\end{tabular}
\caption{VID's of irreducibles in the cases where there is a single $\Gamma$-orbit.}
\label{tab:one_orbit}
\end{table}

The value of $r$, the number of terms, is constrained by Lemma \ref{lem:basic bd}. For two $(q,d)$ pairs, namely $(2,3)$ and $(2,5)$, all values of $r \geq 2$ are admissible, and for a striking reason: there is a single term $T = L_1L_2L_3$ (respectively, $T = L_1L_2L_3Q$) that is divisible by all the operative factors, and hence $T$ can be tacked on to a VIS any number of times without affecting visible irreducibility! Using $T$, we can also concoct VIS's such as
\begin{equation} \label{eq:crank VIS}
  L_1 L_2 L_3 + C,
\end{equation}
expressing translation symmetries of the sets $\II(2,3)$ and $\II(2,5)$. It is a matter of taste whether an expression like \eqref{eq:crank VIS} is truly \emph{visibly} irreducible, insomuch as the irreducibility of the sum rests on the irreducibility of a term $C$ of the same degree! Fortunately, this point is of little consequence for us, since every polynomial admitting a VID will turn out to have one like those in Table \ref{tab:one_orbit}, with each summand involving powers of the operative factors only.

In all other cases, $r \leq 4$. In Table \ref{tab:one_orbit}, we have chosen neither the longest nor the shortest VIS but rather the most symmetrical, minimizing the number of distinct instances and thus minimizing the number of VID's of that shape per irreducible, shown in the last column of the table. In six cases, indeed, we can make the VID unique.

The VIS $L_1 L_2 + L_2 L_3 + L_3 L_1$ tabulated for quadratics over $\FF_2$ is more symmetric than the shape $L_1^2 + L_2 L_3$ discovered above and is the first instance of a \emph{visibly rootless Lagrange interpolation.} Recall that Lagrange interpolation is a general method for computing a polynomial of minimal degree attaining specified values at an arbitrary finite list of points by summing polynomials that vanish at all but one of the given points. In the present context, it is easy to see that the products
\begin{equation*}
  L_1 \cdots \hat{L_i} \cdots L_{q+1},
\end{equation*}
consisting of all but one linear form, form a basis for the homogeneous forms of degree $d = q$ over $\FF_q$. If a form $F$ of degree $q$ has no roots over $\FF_q$, then each basis element has nonzero coefficient, and the sum is a \emph{visibly rootless} expansion of $F$. After $q = 2$, the next case $q = 3$ yields the VIS
\begin{equation*}
  L_1 L_2 L_3 + \alpha L_1 L_2 L_4 + \beta L_1 L_3 L_4 + \gamma L_2 L_3 L_4,
\end{equation*}
of maximal length $r = 4$, which represents the $8$ irreducible cubics over $\FF_3$ just by varying the signs $\alpha, \beta, \gamma \in \FF_3^\cross = \{\pm 1\}$. For degree $d = 4$ onward, rootless polynomials are no longer necessarily irreducible. Rootlessness is a less deep notion than irreducibility, and visibly rootless expansions are easily shown to exist for all polynomials provided that $d$ is large compared to $q$.

\subsection{Corollaries using selected VID's.}
The cases $(q,d) = (2,4)$ and $(q,d) = (3,3)$ of Table \ref{tab:one_orbit} are also noteworthy in yielding the following novel relations among irreducible polynomials in those degrees.

\begin{cor}
The three nonzero linear forms $L_i$ and the three irreducible quartic forms $D_i$ over $\FF_2$ are in a canonical bijection respecting the $\Gamma$-action, given by
\begin{equation*}
  L_1 \mapsto D_1 = L_1^2 L_2 L_3 + Q^2 = L_1^4 + L_2 L_3 Q = L_1^2 Q + L_2^2 L_3^2.
\end{equation*}
\end{cor}
\begin{proof}
It is easy to see that there is only one way for the symmetric group $\Gamma = S_3$ to act transitively on a set of size $3$, up to isomorphism, namely its natural action on three letters. The three letters can be distinguished by their stabilizers, which are the three order-$2$ subgroups of $S_3$.

Consequently, the three $L_i$ and the three $D_i$ comprise isomorphic $\Gamma$-sets. Each of the three VID's listed is invariant under swapping $L_2$ with $L_3$ and thus must represent the unique irreducible quartic $D_1$ fixed by this transposition.
\end{proof}

\begin{cor}
Let $C$ be an irreducible cubic form over $\FF_3$. There is a unique irreducible cubic form $C'$ over $\FF_3$ such that
\begin{itemize}
  \item $C + C'$ is a cube,
  \item $C - C'$ is the product of three distinct linear factors.
\end{itemize}
\end{cor}
\begin{proof}
By Table \ref{tab:one_orbit}, $C$ can be uniquely decomposed as $L_1^3 + L_2 L_3 L_4$. (Here we are noting that every element of $\FF_3$ is a cube, so the term $L_1^3$, which was a priori only a cube up to scaling, is in fact the cube of a linear form $L_1$; and we scale $L_2$, $L_3$, and $L_4$ so that the equality holds.) We see that $C' = L_1^3 - L_2 L_3 L_4$ satisfies the conditions. Conversely, for any $C'$ satisfying the conditions,
\begin{equation*}
  C = \frac{C + C'}{2} + \frac{C - C'}{2}
\end{equation*}
is a VID of $C$ of the shape $L_1^3 + L_2 L_3 L_4$, the two summands clearly being coprime.
\end{proof}

\subsection{Sextics over $\FF_2$.}
There remain two cases of Theorem \ref{thm:main}\ref{it:main all} for which there are multiple $\Gamma$-orbits. These require more work, for instead of merely displaying one VIS, we must find a VIS for each orbit and carefully verify that they indeed cover all the orbits.

The $(2^6 - 2^3 - 2^2 + 2)/6 = 9$ irreducible sextics over $\FF_2$ are a case in point. Three of these form a \emph{special orbit} with point stabilizer of size $2$, represented by the self-reciprocal polynomial $x^6 + x^3 + 1$. The other six form a \emph{generic orbit} with trivial point stabilizer. (There are many ways to verify the sizes and stabilizers of these orbits.) The VIS
\begin{equation*} %\label{eq:sextic_vid_special}
  F_{L_1} = L_1^2 L_2 L_3 Q + C_1 C_2
\end{equation*}
is symmetric under swapping $L_2$ with $L_3$ or $C_1$ with $C_2$; indeed, it has exactly three instances, which must represent the irreducibles in the special orbit. One finds that there is just one other VIS in this degree,
\begin{equation} \label{eq:sextic_vid_generic}
  F_{L_1, L_2, L_3, C_1, C_2} = L_1^2 L_2 C_1 + L_3 Q C_2.
\end{equation}
It is completely asymmetric and yields $12$ instances, some of which must necessarily have the same sum. But who is to say, except by explicit computation, that they are not just additional VID's for the special orbit?

To shed more light on this question, recall Lenstra's proof \cite{Lenstra5} that the unique VIS for $(q,d) = (5,3)$---appearing in the last row of Table \ref{tab:one_orbit}---represents all irreducible cubics over $\FF_5$. His method is quite different from ours: after observing that there are the same number of irreducible cubics as VID's of this shape, he shows directly that no two of the VID's have the same value. Suppose that
\begin{equation} \label{eq:F5}
  L_1 L_2 L_3 + \alpha L_4 L_5 L_6 = \beta L'_1 L'_2 L'_3 + \gamma L'_4 L'_5 L'_6
\end{equation}
for some constants $\alpha, \beta, \gamma \in \FF_5^\cross$ and permutation $\{L'_1,\ldots,L'_6\}$ of $\{L_1,\ldots,L_6\}$. Then observe that some pair of terms, one on each side of \eqref{eq:F5}, must share at least two linear factors. Assume they share exactly two (the other case is trivial): we can reindex so that the relation \eqref{eq:F5} takes the form
\begin{equation} \label{eq:compare_and_factor}
\begin{aligned}
  L_1 L_2 L_3 + \alpha L_4 L_5 L_6 &= \beta L_1 L_2 L_4 + \gamma L_3 L_5 L_6 \\
  L_1 L_2 (L_3 - \beta L_4) &= L_5 L_6 (\gamma L_3 - \alpha L_4).
\end{aligned}
\end{equation}
Now the common value of the two sides is a cubic form divisible by four distinct linear forms $L_1, L_2, L_5, L_6$, which is impossible. At the heart of the proof is a ``compare and factor'' technique \eqref{eq:compare_and_factor}, by which two similar sums are subtracted term by term and proved to be unequal. This ``compare and factor'' method will enable us to avoid brute-force computation of orbit representatives and their VID's.

We are now ready to prove the $(q,d) = (2,6)$ case of Theorem \ref{thm:main}. In fact, we have the following:
\begin{theorem}
The asymmetric shape \eqref{eq:sextic_vid_generic} represents every irreducible sextic over $\FF_2$.
\end{theorem}
\begin{proof}
We ask whether a VID of the asymmetric shape $F_{L_1, L_2, L_3, C_1, C_2}$ can equal one of the symmetric shape $F_{L_1}$. In fact, it can: in an application of the ``compare and factor'' method, the potential equality
\begin{align*}
F_{L_1, L_2, L_3, C_1, C_2} &= F_{L_3} \\
L_1^2 L_2 C_1 + L_3 Q C_2 &= L_1 L_2 L_3^2 Q + C_1 C_2 \\
\intertext{can be written as}
L_3 Q (C_2 + L_1 L_2 L_3) &= C_1 (L_1^2 L_2 + C_2),
\end{align*}
which holds if and only if
\begin{enumerate}
  \item\label{it:crank} $C_1 = C_2 + L_1 L_2 L_3$, and
  \item\label{it:vid} $L_3 Q = L_1^2 L_2 + C_2$.
\end{enumerate}
Equation \ref{it:crank} \emph{always} holds (this was noted above in \eqref{eq:crank VIS}). Equation \ref{it:vid} may or may not hold: $L_1^2 L_2 + L_3 Q$ is a visibly irreducible cubic that may or may not be $C_2$. The six instances of the asymmetric shape \eqref{eq:sextic_vid_generic} that satisfy \ref{it:vid} form a $\Gamma$-orbit that represents the special orbit, each sextic therein occurring twice. As for the six instances that do \emph{not} satisfy \ref{it:vid}, we leave it to the reader to apply the ``compare and factor'' method to eliminate the other possibilities
\begin{equation*}
  F_{L_1, L_2, L_3, C_1, C_2} = F_{L_1} \quad \text{and} \quad F_{L_1, L_2, L_3, C_1, C_2} = F_{L_2},
\end{equation*}
and to conclude that these instances necessarily represent the generic orbit. As these six instances form a $\Gamma$-orbit, we obtain that each irreducible in the generic orbit actually has a unique VID.
\end{proof}
This completes the construction of VID's for the two orbits and the solution of the $(q,d) = (2,6)$ case of Theorem \ref{thm:main}. This may seem like a lot of fuss considering the small number of polynomials involved. But we will now apply the same method to the septimic (7th-degree) case.

\subsection{Septimics over $\FF_2$.}
There are $(2^7 - 2)/7 = 18$ irreducible septimics over $\FF_2$. By Lemma \ref{lem:stab}, all point stabilizers are trivial and there are $3$ orbits, each of size $6$. Ideally, we would seek a VIS that represents all $18$ septimics. There are several VIS's in this degree, but unfortunately, none has more than $3! \cdot 1! \cdot 2! = 12$ instances, due to the limited permutations of the $L_i$'s, $Q$, and the $C_i$'s. (No VIS includes a factor of degree $4$ or greater, as then, even for $r = 2$, the total degree of the two terms would be at least $3(1) + 1(2) + 2(3) + 1(4) = 15 > 7 + 7$.) However, the following VID schema is a union of two shapes that are sufficiently similar to allow the ``compare and factor'' method to work both within and between them.
\begin{theorem} \label{thm:septimic}
Any irreducible septimic over $\FF_2$ is uniquely of the visibly irreducible schema
\begin{equation}\label{eq:septimic_vid}
  L_1^i L_2^{4-i} C_1 + L_3^2 Q C_2
\end{equation}
for some orderings $L_1,L_2,L_3$ and $C_1,C_2$ of the operative forms and some integer $i$, $0 < i < 4$, up to the symmetry that takes $i \mapsto 2-i$ and swaps $L_1$ with $L_2$.
\end{theorem}
\begin{proof}
Because there are $\frac{3! \cdot 2! \cdot 3}{2} = 18$ VID's within the schema, it is enough to show that no two have equal sum. Using the ``compare and factor'' method, we make the following observation: If $F_1,\ldots,F_4$ are quartics (not necessarily irreducible or even distinct) such that
\begin{equation*}
  F_1 C_1 + F_2 C_2 = F_3 C_1 + F_4 C_2
\end{equation*}
but $F_1 \neq F_3$, then from the factorization
\begin{equation*}
  (F_1 - F_3) C_1 = (F_4 - F_2) C_2,
\end{equation*}
we get that $F_1 - F_3 = L C_2$ and $F_4 - F_2 = L C_1$ for some $L \in \{L_1,L_2,L_3\}$. In particular, $F_1 - F_3$ is divisible by exactly one $L_i$ and not by $Q$.

Assume that
\begin{equation*}
  L_1^i L_2^{4-i} C_1 + L_3^2 Q C_2 = {L'_1}^j {L'_2}^{4-j} C'_1 + {L'_3}^2 Q C'_2
\end{equation*}
are two distinct VID's for the same irreducible septimic within the schema \eqref{eq:septimic_vid}, where $\{L'_1,L'_2,L'_3\}$ and $\{C'_1,C'_2\}$ are permutations of $\{L_1,L_2,L_3\}$ and $\{C_1,C_2\}$, respectively. If $C'_2 = C_2$, then the coefficients $F_2$, $F_4$ of $C_2$ on each side are both divisible by $Q$, which is impossible by the observation above. So $C'_2 = C_1$ and $C'_1 = C_2$. But now the difference of the coefficients of $C_1$ on each side is
\begin{equation*}
  F_1 - F_3 = L_1^i L_2^{4-i} + {L'_3}^2 Q.
\end{equation*}
If $L'_3 = L_3$, then $F_1 - F_3$ is divisible by none of the $L_i$. (Indeed, it is a visibly irreducible quartic.) But if $L'$ is one of the other $L_i$, say $L_1$, then $F_1 - F_3$ is divisible by both $L_1$ and $L_3$ since $L_1$ divides both terms and $L_3$ divides neither term. So in no case can $F_1 - F_3$ be divisible by exactly one $L_i$, completing the proof of the theorem and of Theorem \ref{thm:main}\ref{it:main all}.
\end{proof}

\section{Quintics over $\FF_3$.} As promised in Theorem \ref{thm:main}, this is the unique case in which we get VID's for some but not all of the irreducibles of one degree over a field. There are $(3^5 - 3)/5 = 48$ irreducible quintics over $\FF_3$, up to scaling. The group $\Gamma$ has size $24$ (indeed, it is isomorphic to $S_4$, permuting the four $L_i$ freely). The point stabilizer of any irreducible quintic is trivial by Lemma \ref{lem:stab}, so there are two orbits, each of size $24$. In writing a VID, we note that equality holds in Lemma \ref{lem:basic bd} with $r = 2$, so we must have just two terms $f_1$, $f_2$ such that
\begin{equation*}
  f_1 f_2 = \alpha L_1 L_2 L_3 L_4 Q_1 Q_2 Q_3
\end{equation*}
for some $\alpha \in \FF_3^\cross$. There is but a single way, up to reindexing, to split the degree-$10$ polynomial on the right into the product of two quintics: thus there is only a single VIS
\begin{equation*}
  F_{L_1, Q_1, \alpha} = L_1 Q_2 Q_3 + \alpha L_2 L_3 L_4 Q_1.
\end{equation*}
It has $24$ instances (there are $4$ choices for $L_1$, $3$ for $Q_1$, and $2$ for $\alpha$). We conclude that they are the $24$ quintics in one orbit. As there are no other VIS's, the $24$ quintics in the other orbit do not admit a VID. This completes the proof of Theorem \ref{thm:main}.

% section examples (end)

\section{VID's without visible degree.} \label{sec:vid no deg}
To return to our starting point, our notion of VID of a one-variable polynomial included a condition on the degrees of the summands \ref{it:vid deg}, at first seemingly arbitrary, but ultimately explained in terms of the corresponding condition on homogeneous polynomials \ref{it:vid hom} respecting their richer $\Gamma$-symmetry. This article would be incomplete without a few remarks on what would go differently if \ref{it:vid deg} were removed.
 
The proof of Lemma \ref{lem:basic bd} remains unchanged and yields the weaker bound
\begin{equation*}
  dr \geq (r-1) \bigg\lvert \bigcup_{n=1}^{\floor{d/2}} \FF_{q^n} \bigg\rvert
\end{equation*}
in which a summand of $r-1$ is omitted from the right-hand side. Continuing in the manner of Lemma \ref{lem:when bd} yields the same possible $(q,d)$ pairs, with one addition: $(q,d) = (4,2)$. A quadratic over $\FF_4$ cannot have a VID in the sense used throughout this article, but a sum of the shape
\begin{equation} \label{eq:vid_4,2}
  L_1 L_2 + \alpha L_3 L_4,
\end{equation}
omitting the exceptional linear form $L_5(X,Y) = Y$, can be irreducible. Sums of this shape do not have a $\Gamma$-action, but there is an action by the stabilizer $\Gamma_{\infty}$ of $L_5$, which is none other than the group $\GA_1(\FF_4)$ of affine transformations of $\FF_4$. By Lemma \ref{lem:sim tran}\ref{it:st 2}, $\Gamma_{\infty}$ permutes $\II(4,2)$ transitively. There are $\size{\II(4,2)} = (4^2 - 4)/2 = 6$ irreducibles. The shape \eqref{eq:vid_4,2} has $9$ instances (fixing $L_5$), of which $3$ have the value $L_5^2$ up to scaling (one choice of $\alpha$ for each choice of $L_1,L_2,L_3,L_4$). So we have the following uniqueness theorem.
\begin{theorem}
An irreducible quadratic $f(x)$ over $\FF_4$ can be expressed uniquely in the form
\begin{equation*}
  \alpha(x - a)(x - b) + \beta(x - c)(x - d)
\end{equation*}
up to commutativity, where $\{a,b,c,d\} = \FF_4$ and $\alpha, \beta \in \FF_4^\cross$ are distinct scalars.
\end{theorem}

\ignore{
This is far from the only natural generalization of a VID. For instance, for any orderings of the irreducible cubics $C_1, C_2$ and quartics $D_1, D_2, D_3$ over $\FF_2$, the ratio
\begin{equation*}
  \frac{Q^2 C_1 D_1 + C_2 D_2 D_3}{L_1 L_2 L_3} \text{ TODO WRONG}
\end{equation*}
is an irreducible polynomial of degree $8$. Who knows where the search for pleasing representations of irreducible polynomials will end?
}

\section{Conclusion.}
Lenstra's ``compare and factor'' method, coupled with an awareness of the symmetry of the situation, demonstrate for us that VID's of the same degree tend to ``repel'' each other and fill out all irreducibles of a given degree. However, the obtainable families of irreducibles peter out after a finite list, and no case comes close to exceeding the $4 \cdot \size{\II(5,3)} = 160$ cubics in Lenstra's $\FF_5$ example. So the question remains: Is the ``compare and factor'' method, for all its beauty, applicable only to a finite total number of objects? Or are there structures of higher degree, perhaps even in more dimensions, that can be handled in a subtly analogous way?

\section{Acknowledgments.}
This research was supported by the National Science Foundation (grant \#DGE-1646566). I thank my advisor, Manjul Bhargava, for communicating the research question.

%\bibliography{visirred}
%\bibliographystyle{abbrv}

\ignore{
\begin{biog}
\item[Evan O'Dorney] is a graduate student at Princeton University. He completed his undergraduate degree at Harvard College and completed Part III at Cambridge University.
\begin{affil}
Department of Mathematics, Princeton University, Fine Hall, Washington Road, Princeton NJ 08544\\
eodorney@princeton.edu
\end{affil}
\end{biog}
\vfill\eject
}
\end{document}